\documentclass[11pt]{amsart}
\usepackage{amsfonts, amssymb, amsthm, amsmath, amscd}
\usepackage{endnotes}

\newtheorem{prop}{Proposition}[section]

\newtheorem{thm}[prop]{Theorem}
\newtheorem{cor}[prop]{Corollary}

\newtheorem{rem}[prop]{Remark}

\begin{document}

\title
[rational maps on a variety of general type]
{Parametrization of rational maps \\ on a variety 
of general type, \\  and the finiteness theorem}
\author{Lucio Guerra \and Gian Pietro Pirola}
\date{}
\maketitle

\begin{abstract}
{\noindent In a previous paper, we provided some update 
in the treatment of the finiteness theorem for 
rational maps of finite degree from a fixed variety 
to varieties of general type.
In the present paper we present another improvement,
introducing the natural parametrization of maps by means of the
space of linear projections in a suitable projective space,
and this leads to some new insight
in the geometry of the finiteness theorem.

\medskip


}
\end{abstract}

\section*{Introduction}

Let $X$ be an algebraic variety of general type, over the complex field. 
The dominant rational maps of finite degree
$X \dasharrow Y$ to varieties of general type, 
up to birational isomorphisms $Y \dasharrow Y'$,
form a finite set. We call this the {\em finiteness theorem for rational maps
on a variety of general type}. The proof follows from the approach of 
Maehara \cite{M} joined with some recent advances in the theory of
pluricanonical maps, due to Hacon and McKernan \cite{HK} 
and to Takayama \cite{Tak}, \cite{Tak2}.

In our paper \cite{GP}, motivated by the wish of some effective estimate
for the finite number of maps in the theorem, we provided some update 
and refinement in the treatment of the subject.
We brought the rigidity theorem to a general form,
avoiding certain technical restrictions, we pointed out the role of
the canonical volume ${\rm vol}(K_X)$ in bounding the 
rational maps in the finiteness theorem, and we proposed
a new argument leading to a refined version of the theorem.

However, something still not satisfactory was the use of a certain bunch
of subvarieties of Chow varieties as a parameter space for rational maps,
as in Maehara's approach is too. The most natural and simple parameter space 
should be the space of linear projections in a suitable projective space,
already appearing for instance in the work
of Kobayashi and Ochiai \cite{KO}. 

In the present paper we are able to replace the Chow parametrization
with the natural parametrization, and this leads to some new insight
into the geometry of the finiteness theorem. The main result concerns the
structure of the special birational equivalence classes of maps viewed
as unions of connected components of a certain space of linear rational maps, 
see Theorem \ref{connectedcomponent}.
This has as an immediate consequence a better refined finiteness theorem,
see Theorem \ref{finiteness}.
\bigskip

\small \noindent {\em Acknowledgements.}
The first author is partially supported by:
Finanziamento Ricerca di Base 2008 Univ. Perugia.
The second author is partially supported by:
1) INdAM (GNSAGA);
2) FAR 2010 (PV):{\em ``Variet\`{a} algebriche, calcolo
algebrico, grafi orientati e topologici"}.
\normalsize 

\section{Preliminary material}

\subsection*{a. Results on pluricanonical maps}

A recent achievement in the theory of pluricanonical maps is the following
theorem of uniform pluricanonical birational embedding, due to
Hacon and McKernan \cite{HK} and to Takayama \cite{Tak}.

\begin{thm} \label{HKT}
For any dimension $n$ there is some positive integer $r_n$ such that: for
every $n$-dimensional variety $V$ of general type the multicanonical divisor
$r_nK_V$ defines a birational embedding $V \dashrightarrow V' \subset
\mathbb P^M$.
\end{thm}

A basic tool is the canonical volume of a variety, 
the invariant arising in the
asymptotic theory of divisors, see Lazarsfeld's book \cite{L}.
In terms of the canonical volume we have a bound 
\begin{equation} \label{degvol}
\deg V' \leq {\rm vol} (r_{n}K_{V}),
\end{equation}
see \cite{HK}, Lemma 2.2.
Moreover from elementary geometry we have a bound  
\begin{equation} \label{embdim}
M \leq \deg V' +n-1.  
\end{equation}
Note that the embedded variety $V'$ needs not be smooth. 
Intimately related to the theorem above is the following
result, proved in \cite{HK} and in \cite{Tak}.

\begin{thm} \label{HK}
For any dimension $n$ there is some positive number $\epsilon_n$ 
such that every $n$-dimensional variety $V$ of general type has
${\rm vol} (K_{V}) \geq \epsilon_n$.
\end{thm}

For instance, concerning the minimum $r_n$ we know from the classical theory 
that $r_1=3$ and $r_2=5$, and a recent result is that $r_3 \leq 73$, while
concerning the maximum $\epsilon_n$
it is clear that $\epsilon_1 = 2$ and $\epsilon_2 = 1$ and a recent
result is $\epsilon_3 \geq 1/2660$, see J. A. Chen and M. Chen \cite{CC}.
Note that \cite{HK} and \cite{Tak} do not give
explicit bounds for $r_n$ and $\epsilon_n$ in the theorems above.

\subsection*{b. Bounds for the degree of a rational map}

Let $f: X \dasharrow Y$ be a rational map of finite degree
between varieties of general type. 
Because of Theorem \ref{HKT}, taking the $r_n$-canonical
birational models $X'$ and $Y'$ in $\mathbb P^M$ 
(note that $Y'$ lies within the embedding space of $X'$), 
the map $f$ is identified with a {\em linear rational map} $X' \dasharrow Y'$,
a rational map which is the restriction of a linear projection 
$\mathbb P^M \dasharrow \mathbb P^M$.
For a linear map of finite degree the inequality
$\deg f \, \deg Y' \leq \deg X'$ holds.
Using (\ref{degvol}) it follows that
\begin{equation} \label{deg1}
\deg f \leq \deg X' \leq (r_n)^n\, {\rm vol} (K_X).
\end{equation}

A more precise estimate is as follows. For any rational map of finite degree 
the inequality $\deg f \; {\rm vol} (K_Y) \leq {\rm vol} (K_X)$ holds,
see \cite{GP}, Proposition 3.2. Using Theorem \ref{HK} it follows that 
\begin{equation} \label{deg2}
\deg f \leq \dfrac{1}{\epsilon_n}\, {\rm vol} (K_X).
\end{equation}
This bound is sharp for curves, and in this case it reduces
to the usual bound from the Hurwitz formula.

\subsection*{c. Families of rational maps}

Let $T$ be a smooth variety. 
If $X \rightarrow T$ is a relative scheme over $T$,
we denote by $X(t)$ the scheme fibre over $t$, and
by $X_t$ the associated reduced scheme.

A {\em family of varieties}, parametrized by a smooth variety $T$,
is a surjective morphism $X \rightarrow T$, with $X$ a variety,
such that  every scheme fibre $X(t)$ is: $(i)$ irreducible, $(ii)$ generically smooth
(in order to be assigned multiplicity one in the associated algebraic cycle,
see Fulton \cite{F}, Chap. 10),  
and $(iii)$ of dimension equal to the relative dimension of $X$ over $T$, of course. 
When the structure morphism is projective or smooth,
we speak of a family of projective varieties or a family of smooth varieties.

A {\em family of rational maps} is the datum of
a family of varieties $X \rightarrow T$ and 
a relative scheme $X' \rightarrow T$, over the same smooth variety $T$,
and a rational map $f: X \dasharrow X'$, commuting with the structural
projections, which for every $t \in T$ 
restricts to a rational map $f_t: X_t \dasharrow X'_t$.

\subsection*{d. The rigidity theorem}

A family of rational maps {\em  on a fixed variety} $X$ is the datum of a
relative scheme $Y \rightarrow T$, with $T$ smooth, and a rational map
$$f: X\times T \dasharrow Y$$
which is a family of rational maps $f_t: X \dasharrow Y_t$
in the sense of the previous definition.

A {\em trivial family} is one which is obtained as follows.
Let $h: X \dasharrow U$ be a rational map and
let $g: T \times U \dasharrow Y$ be a birational isomorphism
which is a family of birational isomorphisms $g_t: U \dasharrow Y_t$.
Then the composite map 
$$T \times X \overset{1 \times h}{\dasharrow} T \times U 
\overset{g}{\dasharrow} Y$$
is  a trivial family, because
all maps $g_t \circ h$ are birationally equivalent.

Recall that two dominant rational maps 
$f: X \dasharrow Y$ and $f': X \dasharrow Y'$,
defined on the same variety, are {\em birationally equivalent} if there is 
a birational isomorphism $g:Y \dasharrow Y'$ such that  $f' = g \circ f$.  

For projective varieties of general type and 
dominant rational maps of finite degree
there are results of rigidity.

\begin{thm} \label{rigidity}
Let $X$ be a smooth projective variety of general type.
Let $T$ be a smooth variety, let $Y \rightarrow T$ be a family of 
smooth projective varieties of general type, and let
$f: X\times T \dasharrow Y$ be a family of rational maps of finite degree.
Then $f$ is  a trivial family, so all maps $f_t$ are birationally equivalent.
\end{thm}

The rigidity theorem above was proved by Maehara \cite{M}
with some technical restrictions, and has been brought to the
present form in our previous paper \cite{GP}, Theorem 2.1.
More generally, if the family of image varieties is not known 
to be a smooth family, one has the following.

\begin{cor} \label{weakrigidity} 
Let $X$ be a projective variety of general type.
Let $T$ be a smooth variety, let $Y \rightarrow T$ be a  family of 
projective varieties of general type, and let
$f: X\times T \dasharrow Y$ be a family of rational maps of finite degree. 
There is a nonempty open subset $T'$ of $T$ such that the restriction
$f|_{T'}: X\times T' \dasharrow Y|_{T'}$ is a trivial family.
\end{cor}

\section{Graphs and images in a family of maps}

Let $f: X \dasharrow X'$ be a family of rational maps
parametrized by a smooth variety $T$, as in \S 1.c.
Consider the relative product $X \times_{T} X'$ 
and call $p$ and $p'$ the projections to $X$ and $X'$.
{\em Assume now that $X \rightarrow T$ is a projective morphism}. 
Thus $p'$ is a closed map. Then define: \medskip

\begin{tabular}{rl}
$\Gamma$ & the closed graph of $f$ in $X \times_{T} X'$, \\
$Y$ & the closed image of $X$ in $X'$, \\
$C$ & any closed subscheme of $X$ such that 
$X \smallsetminus C \rightarrow T$ is surjective \\ 
& and $f$ is a regular map $X \smallsetminus C \rightarrow Y$, \\
$E$ & the inverse image of $C$ in $\Gamma$.
\end{tabular}
\medskip

\noindent  Note that $p'(\Gamma) = Y$, as $p'$ is a closed map. 

A natural question is whether $\Gamma \rightarrow T$
is the family of closed graphs for the given family of maps, 
more precisely: whether $\Gamma \rightarrow T$ is a family of varieties,
as in \S 1.c, and every reduced fibre $\Gamma_t$ 
coincides with the {\em closed graph} $\Gamma(f_t)$.
A related question is whether $Y \rightarrow T$
is the family of closed images $\overline{f_t(X_t)}$, that is: whether
$Y \rightarrow T$ is a family of varieties and 
every reduced fibre $Y_t$ coincides with the {\em closed image} $\overline{f_t(X_t)}$.
The following equality of reduced schemes holds:
$$\Gamma_t = \Gamma(f_t) \cup E_t$$ 
and from this, applying $p'$, a description of $Y_t$ follows.

\begin{prop} \label{familyofgraphs}
In the setting above, assume that $T$ is a smooth curve.  
$(1)$ There is a nonempty open subset $T'$ of $T$ such that 
$\Gamma|_{T'} \rightarrow T'$ is the family of closed graphs
for the restricted family $f|_{T'}$.
$(2)$ There is a nonempty open subset $T''$ of $T'$ such that moreover
$Y|_{T''} \rightarrow T''$ is the family of closed images 
for the family $f|_{T''}$.
\end{prop}

\begin{proof}
We start with an easy remark.
Let $V \rightarrow T$ be a surjective morphism of varieties,
with irreducible fibres, all of the same dimension.
Then there is a nonempty open subset $T'$ of $T$ such that 
the restriction $V|_{T'} \rightarrow T'$ is a family of varieties.
Now we apply this to the relative varieties $\Gamma$
and $Y$ over the curve $T$.
In order to prove the statement we only need to identify the
reduced fibres $\Gamma_t$ and $Y_t$ for sufficiently general $t$.
This is what we do in the following.

(1) First, we show that $\Gamma_t = \Gamma(f_t)$ holds for every $t$
if $E \rightarrow T$ is a flat morphism.
Recall that this happens if and only if
every irreducible component of $E$ dominates $T$.

Write $\dim X =: n+1$. We have $\Gamma_t = \Gamma(f_t) \cup E_t$. Remark that 
$\dim E < n+1$. Then $\dim E_t < n$ for every $t$, because of flatness.
But all components of $\Gamma_t$ must have dimension $=n$ for every $t$.
Thus $E_t$ is not a component and $\Gamma_t = \Gamma(f_t)$, for every $t$. 
In particular, every $\Gamma_t$ is irreducible of dimension $n$.

In the present situation, the statement follows from the remark in the beginning.
In the general case, by generic flatness, we have that
$E|_{T'} \rightarrow T'$ is flat for some $T'$ and then,
because of the remark, the statement follows.

(2) We know that $Y_t = p'(\Gamma_t)$, and for $t \in T'$ 
we have from (1) that $\Gamma_t = \Gamma(f_t)$ 
and hence $Y_t = \overline{f_t(X_t)}$.
In particular every such $Y_t$ is irreducible,
and necessarily of dimension $= \dim Y -1$.
Because of the remark above, the statement follows.
\end{proof}

In general, the family of graphs needs not exist for the full family of maps,
as is seen later on in Remark \ref{example}.

\section{The varieties of general type in a family}

Using the technique of extension of differentials, from a special fibre
to the total space of the family, we gave in \cite{GP}, \S 1.4, a proof
of the assertion that the property of being a variety of general type
is invariant in a 1-dimensional small deformation, where small refers
to the Zariski topology. Here we point out that the same proof shows
indeed a slightly stronger assertion, to the effect that the same property
'propagates' from a component of a fibre.

\begin{thm}[] \label{generaltype}
Let $T$ be a smooth irreducible curve,
let $Y$ be a variety and let $Y \rightarrow T$ be a
projective morphism. Assume that some fibre
$Y_a$ has an irreducible component $Z$ which
is a variety of general type, and that the restriction
$Y \smallsetminus Y_a \rightarrow T \smallsetminus \{a\}$
is a family of varieties, as in \S {\rm 1.c}. 
Then there is a nonempty open subset $T'$ of $T$
such that $Y_t$ is a variety of general type for $t \in T'$.
\end{thm}

\begin{proof}
Let $V \rightarrow Y$ be a resolution of singularities
such that the strict transform $Z'$ of $Z$ is smooth.  
So $Z'$ is of general type, and \ 
$\dim H^{0}(Z',mK_{Z'}) \geq cm^{n}$ for $m \gg 0$.
Denote by $\pi$ the composite map $V \rightarrow Y \rightarrow T$.
Since $V \rightarrow T$ is generically smooth, and since
$Y \rightarrow T$ is generically a family of varieties,
restricting to some neighborhood of $a$,
we may assume that for every $t \neq a$
the induced map $V_{t} \rightarrow Y_{t}$ is a
resolution of singularities.
As the general $V_t$ is irreducible, it follows that
every $V_t$ is connected, by the Zariski connectedness theorem.

The extension theorem of Takayama \cite{Tak2} applies,
and gives us that there is a surjective restriction homomorphism
$$\begin{array}{ccc}
\pi_{*}\mathcal O_{V}(mK_{V}) \otimes k(a) 
& \longrightarrow &  H^{0}(Z', mK_{Z'}) 
\end{array}.$$

The image $\pi_{*}\mathcal O_{V}(mK_{V})$ is a 
torsion free coherent sheaf on the smooth curve $T$,
hence it is a locally free sheaf.
So the dimension of 
$\pi_{*}\mathcal O_{V}(mK_{V}) \otimes k(t)$
is constant.
For $t=a$ this dimension  is  $\geq cm^{n}$ for $m \gg 0$, 
by what we have seen above.
\newcommand{\localmentelibero}
{let  $f:Y \rightarrow S$ be flat,  $\mathcal F$ on $Y$ be flat over $S$.
if $f_{*} \mathcal{F} \otimes k(t) \rightarrow H^{0}(Y_{t}, \mathcal{F}_{t})$
is surjective then it is an isomorphism, and the same holds in a neighborhood of $t$.
moreover $f_{*} \mathcal{F}$ is locally free in a neighborhood of $t$
[Hartshorne,  p. 290] }

For $t \neq a$, since $mK_{V}|_{V_{t}} = m K_{V_{t}}$,
one has the restriction homomorphism
$$\begin{array}{ccc}
\pi_{*}\mathcal O_{V}(mK_{V}) \otimes k(t) & \longrightarrow & 
H^{0}(V_{t}, \mathcal O_{V_{t}}(mK_{V}|_{V_{t}}))
= H^{0}(V_{t}, mK_{V_{t}})
\end{array}$$
and in a smaller neighborhood of $a$ we may assume that
this is an isomorphism for $t \neq a$.
It follows that
$\dim H^{0}(V_{t},mK_{V_{t}}) 
\geq cm^{n}$ for $m \gg 0$,  hence  $Y_{t}$
is of general type. This holds for every $t$
in a neighborhood of $a$.
\end{proof}

\section{Rigidity and limits}

Another key point in our treatment is a result about
limit maps in a generically trivial family of maps.
The result that we give here is only slightly more general than the
one in our previous paper, and the proof given here is more apparent.

Let $X$ be a projective variety. 
Let $T$ be a smooth irreducible curve,
let $Y \rightarrow T$ be a projective morphism, 
and let $f: T \times X \dasharrow Y$ be a family of rational maps on $X$,
as in \S 1.d. Assume that for every $t \in T$ the rational map 
$f_t : X \dasharrow \overline{f_t(X)}$ is of finite degree $k$.

Assume moreover that the family is {\em generically trivial}, 
as in Corollary \ref{weakrigidity},
i.e. that there is a nonempty open subset $T'$ of $T$ such that
the restriction $f|_{T'}$  is obtained as 
$$T' \times X \overset{1 \times h}{\dasharrow} T' \times U 
\overset{g}{\dasharrow} Y|_{T'}$$
where $h: X \dasharrow U$ is a fixed dominant rational map, and where $g$
is a birational isomorphism which restricts to a birational isomorphism
$g_t: U \dasharrow Y_t$ for every $t \in T'$.
Then $f_{t} = g_{t} \circ h$ for $t \in T'$,
so all these maps are birationally equivalent,
of  degree $\deg(f_{t}) = k = \deg(h)$.

\begin{prop}  \label{rigidityandlimits}
Assume that $f: T \times X \dasharrow Y$ is a family
of rational maps of constant degree $\deg(f_t)=k$, 
and assume that the family is generically trivial, as in the setting above.
Then all maps $f_t$ are in the same birational equivalence class.
\end{prop}

\begin{proof}
Let $a \in T$ be any point, and let us
prove that $f_a$ is in the birational equivalence class
of every $f_t$ with $t \in T'$. 

We may assume that $U$ is a normal variety. 
Recall that for a rational map of varieties over a base curve,
from a normal variety to a variety which is proper over the base,
the exceptional locus is of codimension $\geq 2$,
by the valuative criterion of properness for instance.
It follows that $g: T \times U \dasharrow Y$ 
restricts to a rational map $g_a: U \dasharrow Y_a$.

Since $f = g \circ (1 \times h)$ holds as an equality of
rational maps $T \times X \dasharrow Y$
then there is equality of restrictions  $f_a = g_a \circ h$. 
And since $\deg(f_a) = k = \deg(h)$ then $\deg(g_a) = 1$
and $f_a$ is birationally equivalent to $h$ and to every $f_t$.
\end{proof}

\section{Linear rational maps}

Let $\mathbb P^m = {\rm P}(V^{m+1})$ and let
$X \subseteq \mathbb P^m$ be a non degenerate subvariety,
of dimension $n$. The space of linear maps
$\mathbb P^m \dasharrow \mathbb P^m$ is the projective space
\begin{center}
$\mathbb P^N = {\rm P}({\rm End(V)})$ \ with $N = (m+1)^2-1$.
\end{center}
We denote by $\alpha = \overline\ell$ a point in $\mathbb P^N$
and by $x = \overline v$ a point in $\mathbb P^m$.
The evaluation homomorphism  $(\ell,v) \mapsto \ell(v)$
determines a rational map
$$\mathbb P^N \times X \dasharrow \mathbb P^m$$
and this is the family of linear rational maps
$\alpha : X \dasharrow \mathbb P^m$.
We denote by $\overline{\alpha(X)}$ the closed image 
and by $\Gamma(\alpha)$ the closed graph of the map $\alpha$.

The subscheme $C \subset \mathbb P^N \times X$
defined by $\ell(v)=0$
is the exceptional locus of the rational map above.
Consider the projection $C \rightarrow \mathbb P^N$.
The fibre $C_{\alpha}$ is the trace in $X$ of the center of
the linear projection $\alpha: \mathbb P^m \dasharrow \mathbb P^m$.

\begin{rem} \label{example} \em
The subscheme $\Gamma \subset \mathbb P^N \times X \times \mathbb P^m$
defined by $\ell(v) \wedge w =0$ is the closed graph of the rational map above.
Clearly $\Gamma$ contains $C \times \mathbb P^m$.
The projection $\Gamma \rightarrow \mathbb P^N$
does not define the family of graphs. The fibre is given by
$\Gamma_{\alpha} = \Gamma(\alpha) \cup\, C_{\alpha} \times \mathbb P^m$.
It is clear, just looking at dimensions, that 
$\Gamma_\alpha = \Gamma(\alpha)$ if and only if
$C_{\alpha} = \emptyset$. 
\end{rem}

In $\mathbb P^N$ define the following subsets:
\begin{itemize}
\item[]
$R$ \ \ the subset of all $\alpha$ such that
$\alpha: X \dasharrow \overline{\alpha(X)}$ is of finite degree,
\item[]
$R_k$ \  the subset of all $\alpha \in R$ with $\deg(\alpha)= k$,
\end{itemize}
for every integer $k > 0$.

\begin{prop} \label{constructible}
$(1)$ $R$ is an open subset. $(2)$ $R_k$ is a constructible subset for every $k > 0$. 
\end{prop}

\begin{proof}
(1) In $(\mathbb P^N \times X) \smallsetminus C$ let $U$ be the subset
of pairs $(\alpha,x)$ such that $\dim_{x} \alpha^{-1}(\overline{\alpha(X)}) = 0$.
It is an open subset. In $\mathbb P^N$ the image of $U$ coincides with $R$.
In fact, if $\alpha$ admits some point $x \in X \smallsetminus C_{\alpha}$
which is isolated in its fibre, then its general fibre is of dimension $0$.
As the projection $\mathbb P^N \times X \rightarrow \mathbb P^N$ 
is an open map, $R$ is open in $\mathbb P^N$.
(2) In $\mathbb P^N \times X^{\times k}$ let $U_{k}$ be the subset of sequences 
$(\alpha,x_1,\ldots,x_k) =: (\alpha, \bar x)$ such that every $(\alpha, x_i)$ 
belongs to $U$ and $\alpha(x_1) = \cdots = \alpha(x_k)$ while in the
sequence $(x_1,\ldots,x_k)$ there is no coincidence. 
For every $\alpha \in R$
denote by $U_k(\alpha)$ the fibre of $U_k$ over $\alpha$.
Let $V_k$ be the subset such that $\dim_{(\alpha, \bar x)} U_k(\alpha) = n$.
This is a locally closed subset in $\mathbb P^N \times X^{\times k}$. 
In $\mathbb P^N$ the image ${V_k}'$ of $V_k$
is the locus of $\alpha \in R$ with $\deg \alpha \geq k$. 
In fact, if $\alpha$ admits some sequence $(x_1,\ldots,x_k)$
such that $\dim_{(\bar x)} U_k(\alpha) = n$, as the projection
$U_k(\alpha) \rightarrow X$ has 0-dimensional fibres, then $U_k(\alpha)$
dominates $X$, and hence for a general point $x_1$ the fibre of $\alpha$
contains at least $k$ distinct points $x_1,\ldots,x_k$.
It follows that $R_k$ coincides with ${V_k}' \smallsetminus {V_{k+1}}'$.
\end{proof}

\section{Refined finiteness theorem}

Let $X$ be a smooth projective variety of general type, of dimension $n$.
Let $X' \subset \mathbb P^M$ be the image of $X$ in the $r_n$-canonical 
birational embedding, see Theorem \ref{HKT}. 
Here $M = h^0(X,r_nK_X) -1$ is bounded above in (\ref{embdim}).
Every rational map of finite degree 
$f: X \dasharrow Y$ to a smooth projective variety of general type,
taking the $r_n$-canonical model $Y' \subset \mathbb P^M$, 
gives rise to a linear rational map $\alpha: X' \dasharrow \mathbb P^M$
with $\overline{\alpha(X')} = Y'$. 

In this natural way
the set of birational equivalence classes of rational maps of finite degree
from $X$ to varieties of general type is injected into 
the set of birational equivalence classes 
of linear rational maps of finite degree from $X'$ to $\mathbb P^M$.
Our main result is concerned
with the geometric structure of these special equivalence classes.

\begin{thm} \label{connectedcomponent}
Let $X$ be a smooth projective variety of general type.
A birational equivalence class of rational maps of degree $k$
from $X$ to smooth projective varieties of general type
forms a union of connected components of $R_k$.
\end{thm}

\begin{proof}
Let $\alpha \in R_k$ be such that $\overline{\alpha(X')}$ is of general type. 
Let $T$ be a smooth irreducible curve with a morphism
$T \rightarrow R_k$, that we write as $t \mapsto \alpha_t$,
and with some point $a \in T$ such that $a \mapsto \alpha$.
We claim that all maps $\alpha_t$ are birationally isomorphic to $\alpha$.

Consider the rational map $T \times X' \dasharrow T \times \mathbb P^M$
which represents the family of maps $\alpha_t$.
Let $Y$ be its closed image in $T \times \mathbb P^M$. 
There is a nonempty open subset $T'$ of $T$ such that 
$Y|_{T'} \rightarrow T'$ is the family of closed images,
by Proposition \ref{familyofgraphs}. 

The fibre $Y_a$ contains $\overline{\alpha(X')}$, a variety of general type. 
It follows from Theorem \ref{generaltype}
that, shrinking $T'$ if necessary, we may assume that
for every $t \in T'$ the variety $\overline{\alpha_t(X')}$ is of general type.

Then it follows from Corollary \ref{weakrigidity} to the rigidity theorem
that, shrinking $T'$ again, we may assume that 
the restriction $T' \times X' \dasharrow Y|_{T'}$ is a trivial family.
And then it follows from Proposition \ref{rigidityandlimits} that 
all maps $\alpha_t$ with $t \in T$ are birationally equivalent, as we claimed.

So we reach the conclusion.
Every irreducible curve through $\alpha$ in $R_k$ is the image 
of a smooth irreducible curve $T$ as above, and therefore
is fully contained in the birational equivalence class of $\alpha$.
Therefore every connected curve through $\alpha$ in $R_k$
is fully contained in the birational equivalence class of $\alpha$.
Since $R_k$ is constructible, by Proposition \ref{constructible},
this means that the connected component of $\alpha$ in $R_k$
is contained in the birational equivalence class of $\alpha$.
\end{proof}

The space $R$ admits the stratification 
$\bigsqcup R_k$, where the degree $k$ is bounded above in $(\ref{deg1})$
in terms of the function $r_n$, or in $(\ref{deg2})$ in terms of 
the function $\epsilon_n$. As an immediate consequence of the previous result
we obtain the following refined version of the finiteness theorem,
which improves our previous result \cite{GP}, Theorem 4.3.

\begin{thm}  \label{finiteness}
Let $X$ be a smooth projective variety of general type.
The number of birational equivalence classes of rational maps of finite degree 
from $X$ to smooth projective varieties of general type
is bounded above by the number
of connected components of strata in the stratification
$R = \bigsqcup R_k$.
\end{thm}

We showed in \cite{GP} that the finite number of classes of maps
in the finiteness theorem has an upper bound of the form $B(n,v)$
where $n = \dim (X)$ and $v = {\rm vol}(K_X)$, and that such a function
$B$ can be explicitely computed in terms of the function $r_n$.
This is obtained by means of rather cumbersome computations
with the complexity of a certain bunch of subvarieties of Chow varieties,
that was used as a parameter space for rational maps. We believe that an
analogous computation working with the much simpler parametrization
that has been established in the present paper will lead to a simpler
procedure and to a better result for the function $B$.

\vfill  \small

\noindent {\sc Lucio Guerra }\\
Dipartimento di Matematica e Informatica, Universit\`a di Perugia\\
Via Vanvitelli 1, 06123  Perugia, Italia\\
{\tt guerra@unipg.it}
\bigskip

\noindent {\sc Gian Pietro Pirola}\\
Dipartimento di Matematica, Universit\`a di Pavia\\
via Ferrata 1, 27100 Pavia, Italia\\
{\tt gianpietro.pirola@unipv.it}


\begin{thebibliography}{99}
\bibitem{CC}
J. A. Chen, M. Chen. {\em Explicit birational geometry of 3-folds of general type}, II.
Preprint: arXiv:0810.5044.
\bibitem{F}
W. Fulton. Intersection theory.
Ergebnisse der Mathematik und ihrer Grenzgebiete (3), 2. Springer-Verlag, Berlin, 1984. 
\bibitem{GP}
L. Guerra, G. P. Pirola.
{\em On the finiteness theorem for rational maps on a variety of general type. }
Collect. Math. {\bf 60} (2009), n. 3, 261-276.
\bibitem{HK} C. D. Hacon, J. McKernan.
{\em Boundedness of pluricanonical maps of varieties of general type.}
{Invent. Math.  {\bf 166} (2006), no. 1, 1--25.}
\bibitem{KO}
S. Kobayashi, T. Ochiai.
{\em Meromorphic mappings onto compact complex spaces of general type.}
Invent. Math. {\bf 31} (1975), no. 1, 7--16. 
\bibitem{L} R. Lazarsfeld.
Positivity in algebraic geometry. vv. 2.
Ergebnisse der Mathematik und ihrer Grenzgebiete. (3),
49.  Springer-Verlag, Berlin, 2004.
\bibitem{M}
K. Maehara. {\em A finiteness property of varieties of general type.}
Math. Ann. {\bf 262} (1983), no. 1, 101--123.
\bibitem{Tak}
S.~Takayama. 
{\em Pluricanonical systems on
algebraic varieties of general type.}  Invent. Math. {\bf 165} (2006), no. 3, 551--587.
\bibitem{Tak2}
S.~Takayama. {\em On the invariance and the lower semi-continuity 
of plurigenera of algebraic varieties.}  
J. Algebraic Geom.  {\bf 16}  (2007),  no. 1, 1--18.
\end{thebibliography}
\end{document}